\documentclass{amsart}

\usepackage{amsmath}
\usepackage{amsfonts}
\usepackage{amssymb}
\usepackage{graphicx}
\usepackage{hyperref}
\usepackage{mathrsfs}
\usepackage{pb-diagram}
\usepackage{epstopdf}
\usepackage{amsmath,amsfonts,amsthm,enumerate,amscd,latexsym,curves}
\usepackage{bbm}
\usepackage[mathscr]{eucal}
\usepackage{epsfig,epsf,xypic,epic}
\usepackage{caption}
\usepackage{subcaption}

\newtheorem{theorem}{Theorem}[section]

\newtheorem{lemma}[theorem]{Lemma}
\newtheorem{corollary}[theorem]{Corollary}

\newtheorem{definition}[theorem]{Definition}

\newtheorem{example}[theorem]{Example}

\newtheorem{proposition}[theorem]{Proposition}

\newtheorem{remark}[theorem]{Remark}

\newcommand{\dbar}{\bar{\partial}}

\def\bN{\mathbb{N}}
\def\bZ{\mathbb{Z}}

\def\bC{\mathbb{C}}

\begin{document}

\title[A Fr\"olicher-type inequality for generalized complex manifolds]{A Fr\"olicher-type inequality\\ for generalized complex manifolds}
\author[K. Chan]{Kwokwai Chan}
\address{Department of Mathematics\\ The Chinese University of Hong Kong\\ Shatin\\ Hong Kong}
\email{kwchan@math.cuhk.edu.hk}
\author[Y.-H. Suen]{Yat-Hin Suen}
\address{Department of Mathematics\\ The Chinese University of Hong Kong\\ Shatin\\ Hong Kong}
\email{yhsuen@math.cuhk.edu.hk}


\begin{abstract}
We prove a Fr\"olicher-type inequality for a compact generalized complex manifold $M$, and show that the equality holds if and only if $M$ satisfies the generalized $\partial\dbar$-Lemma. In particular, this gives a unified proof of analogous results in the complex and symplectic cases.
\end{abstract}

\maketitle


\section{Introduction}

In \cite{Angella-Tomassini13}, Angella and Tomassini proved the following deep and beautiful {\em Fr\"olicher-type inequality} for a compact complex manifold $(M,J)$ (\cite[Theorems A and B]{Angella-Tomassini13}):
\begin{align}\label{eq:complex_Frolicher_ineq1}
\sum_{p+q = k} \dim_\bC H^{p,q}_{BC}(M) + \sum_{p+q = k} \dim_\bC H^{p,q}_{A}(M) \geq 2 \dim_\bC H^k_{dR}(M;\bC),
\end{align}
where $H^{\bullet,\bullet}_{BC}(M)$ and $H^{\bullet,\bullet}_A(M)$ are the {\em Bott-Chern} and {\em Aeppli} cohomologies of $M$ respectively. In fact, \eqref{eq:complex_Frolicher_ineq1} follows from a stronger inequality:
\begin{align}\label{eq:complex_Frolicher_ineq2}
\dim_\bC H^{p,q}_{BC}(M) + \dim_\bC H^{p,q}_A(M) \geq \dim_\bC H^{p,q}_{\dbar}(M) + \dim_\bC H^{p,q}_{\partial}(M),
\end{align}
by summing up over $p+q=k$ and applying the classical Fr\"olicher inequality \cite{Frolicher55}:
\begin{align*}
\dim_\bC H^k_{\dbar}(M) \geq \dim_\bC H^k_{dR}(M;\bC).
\end{align*}
They also proved that the equality in \eqref{eq:complex_Frolicher_ineq1} holds for all $k \in \bN$ if and only if $M$ satisfies the $\partial\dbar$-Lemma, hence giving an elegant characterization of the validity of the $\partial\dbar$-Lemma.

In a recent work \cite{Angella-Tomassini14}, the same authors generalized their results to an algebraic and more general setting. As a consequence, they obtained an analogous inequality for generalized complex manifolds - a very interesting class of geometric structures first introduced by Hitchin \cite{Hitchin03} and studied in depth by Gualtieri \cite{Gualtieri11}. More precisely, for a compact generalized complex manifold $(M,\mathcal{J})$ of real dimension $2n$, they proved the following inequality \cite[Theorem 4]{Angella-Tomassini14}:
\begin{align}\label{eq:generalized_Frolicher_ineq1}
\dim_\bC GH^k_{BC}(M) + \dim_\bC GH^k_A(M) \geq \dim_\bC GH^k_{\dbar}(M) + \dim_\bC GH^k_{\partial}(M),
\end{align}
where $GH^k_{BC}(M)$, $GH^k_A(M)$ and $GH^k_{\dbar}(M)$ (and $GH^k_{\partial}(M)$) denote the generalized Bott-Chern, Aeppli and Dolbeault cohomologies of $M$ respectively, and showed that $M$ satisfies the generalized $\partial\dbar$-Lemma if and only if the equality in \eqref{eq:generalized_Frolicher_ineq1} holds for all $k\in \bZ$ and the corresponding Hodge and Fr\"olicher spectral sequences degenerate at $E_1$.

For a compact symplectic manifold $(M,\omega)$, their result specializes to the inequality \cite[Theorem 3]{Angella-Tomassini14}:
\begin{align}\label{eq:symplectic_Frolicher_ineq}
\dim_\bC H^k_{BC}(M) + \dim_\bC H^k_{A}(M) \geq 2 \dim_\bC H^k_{dR}(M;\bC),
\end{align}
where $H^\bullet_{BC}(M)$ and $H^\bullet_A(M)$ are the symplectic versions of the Bott-Chern and Aeppli cohomologies respectively, as defined by Tseng and Yau \cite{Tseng-Yau12a, Tseng-Yau12b}, and the equality in \eqref{eq:symplectic_Frolicher_ineq} holds for all $k \in \bN$ if and only if $M$ satisfies the $dd^\Lambda$-Lemma - the symplectic counterpart of the $\partial\dbar$-Lemma.

In this short note, using Hodge theory and a bigrading on differential forms introduced by Cavalcanti \cite{Cavalcanti06}, we give a different, and more geometric, proof of the inequality \eqref{eq:generalized_Frolicher_ineq1} and obtain a stronger statement in the equality case:
\begin{theorem}\label{thm:main_thm}
Let $(M,\mathcal{J})$ be a compact generalized complex manifold of real dimension $2n$. Then, for every $k\in [-n,n] \cap \bZ$, the following inequality between the dimensions of the generalized Bott-Chern and the generalized Dolbeault cohomologies of $M$ holds:
\begin{align}\label{eq:generalized_Frolicher_ineq2}
\dim_\bC GH^k_{BC}(M) \geq \dim_\bC GH^k_{\dbar}(M).
\end{align}
Moreover, the equality in \eqref{eq:generalized_Frolicher_ineq2} holds for all $k\in [-n,n] \cap \bZ$ if and only if $M$ satisfies the generalized $\partial\dbar$-Lemma.
\end{theorem}

\begin{remark}
By using Hodge theory, we see that $\dim_\bC GH^k_A(M) = \dim_\bC GH^k_{BC}(M)$ and $\dim_\bC GH^k_{\partial}(M) = \dim_\bC GH^k_{\dbar}(M)$ (see Section \ref{subsec:Hodge} and Proposition \ref{prop:dim_equalities_generalized}), so the inequality \eqref{eq:generalized_Frolicher_ineq2} we proved is in fact equivalent to \eqref{eq:generalized_Frolicher_ineq1}. However, again because of our use of Hodge theory, we are able to remove the condition that the Hodge and Fr\"olicher spectral sequences degenerate at $E_1$ in proving the validity of the generalized $\partial\dbar$-Lemma when the equality holds.
\end{remark}

\begin{remark}
As in \cite{Angella-Tomassini14}, our results still hold when the generalized complex structure is twisted by a 3-form $H$; we leave the straightforward generalization of our proofs to the reader.
\end{remark}

Before going into the details of the proof, let us explain what Theorem \ref{thm:main_thm} means in the two extreme cases.
In the complex case, we have

\begin{corollary}
Let $(M,J)$ be a compact complex manifold. Then, for every $k \in [-n,n] \cap \bZ$, the following inequality holds:
\begin{align}\label{eq:complex_Frolicher_ineq3}
\sum_{p-q = k} \dim_\bC H^{p,q}_{BC}(M) \geq \sum_{p-q = k} \dim_\bC H^{p,q}_{\dbar}(M).
\end{align}
Moreover, the equality in \eqref{eq:complex_Frolicher_ineq3} holds for all $k \in [-n,n] \cap \bZ$ (or equivalently, $\dim_\bC H^{p,q}_{BC}(M) = \dim_\bC H^{p,q}_{\dbar}(M)$ for all $p,q \in \bN$) if and only if $M$ satisfies the $\partial\dbar$-Lemma.
\end{corollary}
\begin{proof}
When the generalized complex structure is an ordinary complex structure, the generalized $\partial\dbar$-Lemma is equivalent to the ordinary $\partial\dbar$-Lemma.
Also, we have
$$GH^k_{\dbar}(M) = \bigoplus_{p-q=k}H^{p,q}_{\dbar}(M),\quad
GH^k_{BC}(M) = \bigoplus_{p-q=k}H^{p,q}_{BC}(M).$$
So if the $\partial\dbar$-Lemma holds, we have $\dim_\bC H^{p,q}_{BC}(M) = \dim_\bC H^{p,q}_{\dbar}(M)$ for all $p,q \in \bN$ and hence equality in \eqref{eq:complex_Frolicher_ineq3} holds for all $k \in [-n,n] \cap \bZ$. Conversely, if the equality holds for all $k \in [-n,n] \cap \bZ$, then the $\partial\dbar$-Lemma holds by Theorem \ref{thm:main_thm}.
\end{proof}

\begin{remark}
The inequality \eqref{eq:complex_Frolicher_ineq2} of Angella and Tomassini is more refined than \eqref{eq:complex_Frolicher_ineq3} above. This can be seen as follows: Conjugation and the Hodge star operator associated to a given Hermitian metric induce isomorphisms between cohomologies which give the following equalities:
\begin{equation}\label{eq:dim_equalities_complex}
\begin{split}
h^{p,q}_{BC} & = h^{q,p}_{BC} = h^{n-p,n-q}_A = h^{n-q,n-p}_A;\\
h^{p,q}_{\dbar} & = h^{q,p}_{\dbar} = h^{n-p,n-q}_\partial = h^{n-q,n-p}_\partial
\end{split}
\end{equation}
for all $p,q \in \bN$, where $h^{p,q}_\sharp := \dim_\bC H^{p,q}_\sharp(M)$ for $\sharp \in \{\partial, \dbar, A, BC\}$, so that \eqref{eq:complex_Frolicher_ineq2} can be rewritten as
$$h^{p,q}_{BC} + h^{n-q,n-p}_{BC} \geq h^{p,q}_{\dbar} + h^{n-q,n-p}_{\dbar},$$
from which \eqref{eq:complex_Frolicher_ineq3} follows by summing over $p-q = k$.

In retrospect, the fact that we have a stronger inequality (namely, \eqref{eq:complex_Frolicher_ineq2}) in the complex case is because there is a natural bigrading on differential forms; in contrast, we only have an artificial bigrading on differential forms in the generalized complex case (see Section \ref{subsec:bigrading}).
\end{remark}

In the symplectic case, we have
\begin{corollary}[Theorem 4.4 in \cite{Angella-Tomassini14}]
Let $(M,\omega)$ be a compact symplectic manifold. Then, for every $k \in \bN$, the following inequality holds:
\begin{align}\label{eq:symplectic_Frolicher_ineq2}
\dim_\bC H^k_{BC}(M) \geq \dim_\bC H^k_{dR}(M).
\end{align}
Moreover, the equality in \eqref{eq:symplectic_Frolicher_ineq2} holds for all $k \in \bN$ if and only if $M$ satisfies the $dd^\Lambda$-Lemma.
\end{corollary}
\begin{proof}
When the generalized complex structure is an ordinary symplectic structure, we have the isomorphisms
$$GH^k_{\dbar}(M)\cong H^{n-k}_{dR}(M;\bC),\quad GH^k_{BC}(M)\cong H^{n-k}_{BC}(M),\quad GH^k_{A}(M)\cong H^{n-k}_{A}(M).$$
Also, the generalized $\partial\dbar$-Lemma is equivalent to the $dd^{\Lambda}$-Lemma. So the results follow immediately from Theorem \ref{thm:main_thm}.
\end{proof}

\begin{remark}
The inequalities \eqref{eq:symplectic_Frolicher_ineq} and \eqref{eq:symplectic_Frolicher_ineq2} are equivalent because the Hodge star operator $*:\Omega^k(M)\rightarrow\Omega^{2n-k}(M)$ (associated to any compatible metric) induces an isomorphism
$$H^k_{BC}(M)\cong H^{2n-k}_A(M),$$
while the Lefschetz operator induces another isomorphism
$$L^{n-k}:H^k_{BC}(M)\cong H^{2n-k}_{BC}(M),$$
(see \cite{Tseng-Yau12a}), so that in fact $\dim_\bC H^k_{BC}(M) = \dim_\bC H^k_A(M)$ for any $k \in \bN$.
\end{remark}

\begin{remark}
Notice that we have
$$\sum_{p-q=k}\dim_\bC H^{p,q}_{\sharp}(M) = \sum_{p+q=n-k}\dim_\bC H^{n-p,q}_{\sharp}(M),$$
for $\sharp\in\{\dbar,BC\}$ and for any $k \in [-n,n] \cap \bZ$. So from the point of view of mirror symmetry, the inequalities \eqref{eq:complex_Frolicher_ineq3} and \eqref{eq:symplectic_Frolicher_ineq2} are mirror to each other.
\end{remark}


\section*{Acknowledgment}
The work of the first author described in this paper was substantially supported by grants from the Research Grants Council of the Hong Kong Special Administrative Region, China (Project No. CUHK404412 $\&$ CUHK400213).

\section{Basics of generalized complex geometry}

In this section, we briefly review the notions and several basic results which we will need in the proof of our main theorem. Basically we will follow the notations in \cite{Cavalcanti06, Gualtieri11}.

\subsection{Generalized Bott-Chern and Aeppli cohomologies and the generalized $\partial\dbar$-Lemma}
Let $M$ be a compact manifold of real dimension $2n$. Recall that a {\em generalized almost complex structure} on $M$ is an endomorphism $\mathcal{J} \in \text{End}(T\oplus T^*)$ satisfying $\mathcal{J}^2 = -1$.

The {\em canonical line bundle} $U^n$ of $(M,\mathcal{J})$ is defined as the complex pure spinor line bundle $U^n \subset \wedge^\bullet T^*\otimes\mathbb{C}$ annihilated by the $i$-eigenbundle $L$ of $\mathcal{J}$:
$$U^n:=\left\{ \varphi \in \wedge^\bullet T^*\otimes\mathbb{C} \mid L\cdot\varphi=0 \right\}.$$
By putting $U^k:=\wedge^{n-k}\overline{L}\cdot U^n$ for $k \in [-n,n] \cap \bZ$, we have a decomposition
$$\wedge^\bullet T^*\otimes\mathbb{C}=U^{-n}\oplus\cdots\oplus U^n,$$
which induces a $\bZ$-grading on differential forms.
Indeed, the space $U^k$ is the $ik$-eigenbundle of $\mathcal{J}$ acting on the spin representation.

Consider the operators
\begin{align*}
\partial = \pi_{k+1} \circ d :U^k\rightarrow U^{k+1},\\
\dbar = \pi_{k-1} \circ d :U^k\rightarrow U^{k-1},
\end{align*}
where $d$ is the exterior derivative and $\pi_k$ denotes the projection onto $U^k$.
By \cite[Theorem 3.15]{Gualtieri11}, a generalized almost complex structure $\mathcal{J}$ is {\em integrable}, i.e. $[L,L]\subset L$, if and only if $d = \partial + \dbar$. From now on, we will assume that $\mathcal{J}$ is integrable so that $(M,\mathcal{J})$ is a compact {\em generalized complex manifold}.

The property that $d^2=0$ is equivalent to
$$\partial^2 = 0,\quad \dbar^2 = 0,\quad \partial\dbar = -\dbar\partial.$$
This in turn implies that the operator $d^{\mathcal{J}}:U^k\rightarrow U^{k+1}\oplus U^{k-1}$ defined by
$$d^{\mathcal{J}}:=-i(\partial-\dbar)$$
satisfies
$$dd^{\mathcal{J}}=-2i\partial\dbar:U^k\rightarrow U^k.$$

\begin{definition}[\cite{Cavalcanti06}]\label{defn:ddbar-lemma}
A generalized complex manifold $(M,\mathcal{J})$ is said to satisfy the {\em generalized $\partial\dbar$-Lemma} (or equivalently, the {\em $dd^{\mathcal{J}}$-Lemma}) {\em on $U^k$} if
$$\text{ker}(\dbar)\cap \text{im}(\partial)\cap U^k = \text{ker}(\partial)\cap \text{im}(\dbar)\cap U^k = \text{im}(\partial\dbar)\cap U^k.$$
We say $(M,\mathcal{J})$ satisfies the {\em generalized $\partial\dbar$-Lemma} if it satisfies the generalized $\partial\dbar$-Lemma on $U^k$ for all $k\in [-n,n]\cap\mathbb{Z}$.
\end{definition}

\begin{definition}\label{defn:cohomologies}
The {\em generalized Bott-Chern, Aeppli and Dolbeault cohomologies} of $(M,\mathcal{J})$ are defined respectively as the following $\bZ$-graded algebras:
$$GH^\bullet_{\partial}(M):=\frac{\text{ker }\partial}{\text{im }\partial},\quad GH^\bullet_{\dbar}(M):=\frac{\text{ker }\dbar}{\text{im }\dbar}$$
and
$$GH^\bullet_{BC}(M):=\frac{\text{ker }d}{\text{im }\partial\dbar},\quad GH^\bullet_A(M):=\frac{\text{ker }\partial\dbar}{\text{im }\partial + \text{im }\dbar}.$$
\end{definition}
\begin{remark}
Since $d^{\mathcal{J}}:=-i(\partial-\dbar)$, the definitions of the generalized Bott-Chern and Aeppli cohomologies here coincide with that in \cite{Tseng-Yau14}.
\end{remark}
The aim of this note is to investigate the relations between these cohomologies and validity of the generalized $\partial\dbar$-Lemma.

\begin{example}
If the generalized complex structure is given by an ordinary complex structure $J$ on $M$, namely, if
$$\mathcal{J}=
\begin{pmatrix}
-J & 0
\\0 & J^*
\end{pmatrix},$$
then $d^{\mathcal{J}}$ is given by the $d^c$ operator defined by $d^c=-i(\partial-\dbar)$, and
the generalized $\partial\dbar$-Lemma is equivalent to the usual $\partial\dbar$-Lemma.

In this case, we have
$$U^k=\bigoplus_{p-q=k}\Omega^{p,q}(M),$$
for $k \in [-n,n] \cap \bZ$ and
the splitting $d=\partial+\dbar$ is exactly the usual Dolbeault splitting. Hence we have
$$GH^k_{\sharp}(M) = \bigoplus_{p-q=k} H_{\sharp}^{p,q}(M),$$
for $\sharp\in\{\partial,\dbar,A,BC\}$.
\end{example}

\begin{example}
When the generalized complex structure $\mathcal{J}$ is an ordinary symplectic structure $\omega$, namely, when
$$\mathcal{J}=
\begin{pmatrix}
0 & -\omega^{-1}
\\\omega & 0
\end{pmatrix},$$
the differential operator $d^{\mathcal{J}}$ is given by the symplectic adjoint $d^{\Lambda} = \Lambda d - d \Lambda$ introduced by Brylinski \cite{Brylinski88}, where $\Lambda$ is the interior product with the bivector $-\omega^{-1}$. The generalized $\partial\dbar$-Lemma is equivalent to the $dd^\Lambda$-Lemma.

In \cite[Theorems 2.2 and 2.3]{Cavalcanti06}, it was shown that the natural map
$$\varphi:\Omega^k(M)\to U^{n-k},\quad \varphi(\alpha)=e^{i\omega}e^{\frac{\Lambda}{2i}}\alpha$$
gives an isomorphism $\Omega^k(M)\cong U^{n-k}$ satisfying
$$\varphi(d\alpha)=\dbar(\varphi(\alpha)),\quad \varphi(d^{\Lambda}\alpha)=-2i\partial(\varphi(\alpha)).$$
This implies that we have the following isomorphisms
\begin{align*}
GH^k_{\dbar}(M) & \cong H^{n-k}_{dR}(M),\quad GH^k_{\partial}(M)\cong H^{n-k}_{d^{\Lambda}}(M),\\
GH^k_{BC}(M) & \cong H^{n-k}_{BC}(M),\quad GH^k_{A}(M)\cong H^{n-k}_{A}(M).
\end{align*}
\end{example}

\subsection{Generalized metric and Hodge theory}\label{subsec:Hodge}

Since the exterior derivative commutes with the conjugation $C:U^k\rightarrow U^{-k}$, the following diagram
\begin{equation*}
\xymatrix{
U^k \ar[d]_{\partial} \ar@{->}[r]^{C} & {U^{-k}} \ar[d]^{\dbar}
\\U^{k+1} \ar@{->}[r]^{C} & U^{-k-1}
}
\end{equation*}
commutes. Hence conjugation gives rise to the isomorphisms:
\begin{align}\label{eq:isom_conjugation}
GH^k_{\partial}(M) \cong GH^{-k}_{\dbar}(M),\quad GH^k_{BC}(M) \cong GH^{-k}_{BC}(M),\quad GH^k_A(M)\cong GH^{-k}_A(M),
\end{align}
for $k \in [-n,n] \cap \bZ$.

On the other hand, recall that a {\em generalized metric} $\mathcal{G}$ on $M$ is defined as a self-adjoint orthogonal transformation $\mathcal{G} \in \text{End}(T\oplus T^*)$ such that $\langle \mathcal{G}v,v \rangle > 0$ for $v \in T\oplus T^* \setminus \{0\}$, where $\langle \cdot, \cdot \rangle$ denotes the natural pairing on $T\oplus T^*$. By choosing a generalized metric $\mathcal{G}$ on $(M,\mathcal{J})$ compatible with the generalized complex structure $\mathcal{J}$ (meaning that $[\mathcal{G},\mathcal{J}]=0$), we can define the generalized Hodge star operator $*_G$ as
$$*_G\alpha=(-1)^{|\alpha|(n-1)}\tau\cdot\alpha,$$
where $\tau:=e_1\cdots e_n \in \text{Clif}(T\oplus T^*)$ and $\{e_i\}$ is an orthonormal basis of the $+1$-eigenbundle of $\mathcal{G}$.
It is not hard to check that $*_G$ preserves the decomposition $\Omega^\bullet(M)=\bigoplus_{k=-n}^nU^k$, i.e. $*_G$ maps $U^k$ to $U^k$ (cf. \cite[Lemma 3.1]{Cavalcanti06}).

Now we define $\partial^*=-\bar{*}_G\partial\bar{*}_G^{-1}$ and $\dbar^*=-\bar{*}_G\dbar\bar{*}_G^{-1}$, where $\bar{*}_G\alpha := *_G\bar{\alpha}$ and consider the various Laplacian operators defined by
\begin{align*}
\Delta_{\partial} & = \partial\partial^*+\partial^*\partial,\quad \Delta_{\dbar} = \dbar\dbar^*+\dbar^*\dbar,\\
\Delta_{BC}&=(\partial\dbar)(\dbar^*\partial^*)+(\dbar^*\partial^*)(\partial\dbar)+(\dbar^*\partial)(\partial^*\dbar)+(\partial^*\dbar)(\dbar^*\partial)+\dbar^*\dbar+\partial^*\partial,
\\
\Delta_A &= (\dbar^*\partial^*)(\partial\dbar)+(\partial\dbar)(\dbar^*\partial^*)+(\partial\dbar^*)(\dbar\partial^*)+(\dbar\partial^*)(\partial\dbar^*)
+\partial\partial^*+\dbar\dbar^*.
\end{align*}
Since these are all elliptic operators and $GH^k_{\sharp}(M) \cong \text{ker }\Delta_\sharp$ for $\sharp \in \{\partial,\dbar,A,BC\}$ (cf. \cite{Schweitzer07}), the cohomology groups in Definition \ref{defn:cohomologies} are all finite dimensional and each cohomology class has a unique harmonic representative.

\begin{lemma}
We have the commutation relations
$$\bar{*}_G\Delta_{\partial}=\Delta_{\partial}\bar{*}_G,\quad \bar{*}_G\Delta_{\dbar}=\Delta_{\dbar}\bar{*}_G,\quad \bar{*}_G\Delta_{BC}=\Delta_A\bar{*}_G.$$
In particular, for $k\in [-n,n]\cap\mathbb{Z}$, we have the following isomorphisms induced by $\bar{*}_G$:
$$GH^k_{\dbar}(M) \cong GH^{-k}_{\dbar}(M),\quad GH^k_{\partial}(M) \cong GH^{-k}_{\partial}(M),\quad GH^k_{BC}(M)\cong GH^{-k}_A(M).$$
As a whole, we get the following equalities in dimensions:
$$Gh^k_{\dbar}=Gh^{-k}_{\dbar}=Gh^k_{\partial}=Gh^{-k}_{\partial},$$
$$ Gh^k_{BC}=Gh^{-k}_A=Gh^k_A=Gh^{-k}_{BC},$$
where we denote $Gh^k_{\sharp}:=\dim_{\mathbb{C}} GH^k_{\sharp}(M)$ for $\sharp\in\{\partial,\dbar,A,BC\}$.
\end{lemma}
\begin{proof}
We have
$$(\partial\dbar)(\dbar^*\partial^*)\bar{*}_G=\bar{*}_G(\partial^*\dbar^*)(\dbar\partial),\quad
(\dbar^*\partial^*)(\partial\dbar)\bar{*}_G=\bar{*}_G(\dbar\partial)(\partial^*\dbar^*),$$
$$(\dbar^*\partial)(\partial^*\dbar)\bar{*}_G=\bar{*}_G(\dbar\partial^*)(\partial\dbar^*),\quad
(\partial^*\dbar)(\dbar^*\partial)\bar{*}_G=\bar{*}_G(\partial\dbar^*)(\dbar\partial^*),$$
and
$$\dbar^*\dbar\bar{*}_G=\bar{*}_G\dbar\dbar^*,\quad \partial^*\partial\bar{*}_G=\bar{*}_G\partial\partial^*,$$
$$\dbar\dbar^*\bar{*}_G=\bar{*}_G\dbar^*\dbar,\quad \partial\partial^*\bar{*}_G=\bar{*}_G\partial^*\partial.$$
Also $\partial\dbar=-\dbar\partial$, so we have $\partial^*\dbar^*=-\dbar^*\partial^*$. It is now straightforward to see that the desired commutation relations
\begin{align*}
\bar{*}_G\Delta_{\partial}=\Delta_{\partial}\bar{*}_G,\quad \bar{*}_G\Delta_{\dbar}=\Delta_{\dbar}\bar{*}_G,\quad \bar{*}_G\Delta_{BC}=\Delta_A\bar{*}_G
\end{align*}
hold.
\end{proof}

Together with the isomorphisms \eqref{eq:isom_conjugation} induced by conjugation, we obtain the following equalities between dimensions of cohomologies, analogous to those \eqref{eq:dim_equalities_complex} in the complex case:
\begin{proposition}\label{prop:dim_equalities_generalized}
For a compact generalized complex manifold $(M,\mathcal{J})$ of real dimension $2n$, we have
\begin{equation}\label{eq:dim_equalities_generalized_1}
\begin{split}
Gh^k_{\dbar} & = Gh^{-k}_{\dbar} = Gh^k_{\partial} = Gh^{-k}_{\partial},\\
Gh^k_{BC} & = Gh^{-k}_A = Gh^k_A = Gh^{-k}_{BC},
\end{split}
\end{equation}
$k\in [-n,n]\cap\mathbb{Z}$.
\end{proposition}

\subsection{A bigrading on differential forms}\label{subsec:bigrading}

To analyze the equality case of our inequality, we need one more ingredient, namely, a bigrading on the complex of differential forms introduced by Cavalcanti \cite[Section 5]{Cavalcanti06}, who mimicked the constructions of Goodwillie \cite{Goodwillie85} and Brylinski \cite{Brylinski88}.

We consider a formal element $\beta$ of degree 2, and the {\em canonical complex} $\mathcal{A}^{\bullet,\bullet}$ defined by
$$\mathcal{A}^{p,q} = U^{p-q}\beta^q$$
for $p,q \in \bN$.
We extend the exterior derivative $d$ by
$$d^{\beta}(\alpha\beta^k)=(\partial\alpha)\beta^k+(\dbar\alpha)\beta^{k+1}.$$
Then $d^{\beta}$ splits into two components
$$\partial^{\beta}:\mathcal{A}^{p,q}\rightarrow\mathcal{A}^{p+1,q},\quad\dbar^\beta:\mathcal{A}^{p,q}\rightarrow\mathcal{A}^{p,q+1}.$$

\begin{definition}
A generalized complex manifold is said to satisfy the {\em $\partial^{\beta}\dbar^{\beta}$-Lemma on $\mathcal{A}^{p,q}$} if
$$ker(\dbar^{\beta})\cap Im(\partial^{\beta})\cap\mathcal{A}^{p,q}=ker(\partial^{\beta})\cap Im(\dbar^{\beta})\cap\mathcal{A}^{p,q}=Im(\partial^{\beta}\dbar^{\beta})\cap\mathcal{A}^{p,q}.$$
\end{definition}

The following lemma is immediate:
\begin{lemma}\label{lem:ddbar_vs_ddbarbeta}
The generalized $\partial\dbar$-Lemma holds if and only if the $\partial^{\beta}\dbar^{\beta}$-Lemma holds on $\mathcal{A}^{p,q}$ for all $p,q\in \bN$.
\end{lemma}

We can also define the generalized Bott-Chern, Aeppli and Dolbeault cohomologies with respect to the operators $\partial^\beta, \dbar^\beta$ in exactly the same way as in Definition \ref{defn:cohomologies}. Recall that the differential complex $\Omega^\bullet(M)$ sits inside $\mathcal{A}^{\bullet,\bullet}$ via the map (\cite[Section 5]{Cavalcanti06})
$$\tau: U^k \to \bigoplus_{p-q=k}\mathcal{A}^{p,q} = \bigoplus_{p-q=k}U^{p-q}\beta^q$$
defined by
$$\tau(\alpha)=\sum_{q \in \bZ} \alpha\beta^q.$$
The sum is in fact finite since one can only write $k$ in the form $p-q$ in finitely many ways. One can check that
$$\tau(\partial\alpha)=\partial^{\beta}\tau(\alpha),\quad \tau(\dbar\alpha)=\dbar^{\beta}\tau(\alpha).$$
Since each piece $\mathcal{A}^{p,q}$ is nothing but (isomorphic to) $U^{p-q}$, we obtain the isomorphisms
$$H^{p,q}_{\sharp^{\beta}}(M)\cong GH^{p-q}_{\sharp}(M),$$
for $\sharp\in\{\partial,\dbar,BC,A\}$.

Now let $\mathcal{A}^k:=\bigoplus_{p+q=k}\mathcal{A}^{p,q}$. The bigrading naturally gives the bounded filtrations:
\begin{align*}
F^p\mathcal{A} & := \bigoplus_{q}\bigoplus_{r\geq p}\mathcal{A}^{r,q},\\
F^p\mathcal{A}^m & := \bigoplus_{r\geq p}\mathcal{A}^{r,m-r},
\end{align*}
from which we obtain the {\em canonical spectral sequence}, which converges to the cohomology of $d^\beta$. The first term of the spectral sequence is given by $E_1^{p,q} = GH^{p-q}_{\dbar}(M)$.
If the generalized $\partial\dbar$-Lemma holds, then this spectral sequence degenerates at $E_1$;
conversely, if the canonical spectral sequence degenerates at $E_1$ and the decomposition of $\Omega^\bullet(M)$ by $U^k$ induces a decomposition in cohomology, then the generalized $\partial\dbar$-Lemma holds (\cite{DGMS75}; for a more detailed discussion, see \cite[Section 5]{Cavalcanti06}).
This is a characterization of the generalized $\partial\dbar$-Lemma in terms of cohomological decomposition.

\section{Proof of Theorem~\ref{thm:main_thm}}

We are now ready to prove our main results.

\begin{theorem}
For a compact generalized complex manifold $(M,\mathcal{J})$ of real dimension $2n$, we have the inequality
$$Gh^k_{BC} \geq Gh^k_{\dbar}$$
for all $k \in [-n,n] \cap \bZ$.
\end{theorem}
\begin{proof}
We follow the same strategy as in \cite{Angella-Tomassini13}, namely, we define
\begin{align*}
A^{\bullet} := \frac{\text{im }\dbar \cap \text{im }\partial}{\text{im }\partial\dbar},\quad
B^{\bullet} := \frac{\text{ker }\dbar \cap \text{im }\partial}{\text{im }\partial\dbar},\quad
C^{\bullet} := \frac{\text{ker }\partial\dbar}{\text{ker }\dbar+\text{im }\partial},
\end{align*}
and
\begin{align*}
D^{\bullet} := \frac{\text{im }\dbar \cap \text{ker }\partial}{\text{im }\partial\dbar},\quad
E^{\bullet} := \frac{\text{ker }\partial\dbar}{\text{ker }\partial + \text{im }\dbar},\quad
F^{\bullet} := \frac{\text{ker }\partial\dbar}{\text{ker }\dbar + \text{ker }\partial}.
\end{align*}
Note that conjugation induces the isomorphisms
$$D^k \cong B^{-k},\quad E^k \cong C^{-k},$$
and hence gives the equalities
$$d^k = b^{-k},\quad e^k = c^{-k}.$$

On the other hand, as in Varouchas \cite{Varouchas86}, we have the following exact sequences:
$$0 \to A^{\bullet} \to B^{\bullet} \to GH^{\bullet}_{\dbar}(M) \to GH^{\bullet}_{A}(M) \to C^{\bullet} \to 0$$
and
$$0 \to D^{\bullet} \to GH^{\bullet}_{BC}(M) \to GH^{\bullet}_{\dbar}(M) \to E^{\bullet} \to F^{\bullet} \to 0.$$
Each of the cohomologies is finite dimensional and, in particular, we have the following equalities between dimensions
\begin{align*}
Gh^k_{A} & = Gh_{\dbar}^k + a^k + c^k - b^k,\\
Gh^k_{BC} & = Gh^k_{\dbar} + d^k + f^k - e^k,
\end{align*}
for any $k \in [-n,n] \cap \bZ$.

Using these equalities together with \eqref{eq:dim_equalities_generalized_1}, we have
\begin{align*}
2 Gh^k_{BC}
& = Gh^k_{BC} + Gh^{-k}_{A}\\
& = Gh^k_{\dbar} + d^k + f^k - e^k + Gh^{-k}_{\dbar} + a^{-k} + c^{-k} - b^{-k}\\
& = Gh^k_{\dbar} + Gh^{-k}_{\dbar} + f^k + a^{-k}\\
& \geq Gh^k_{\dbar} + Gh^{-k}_{\dbar} = 2 Gh^k_{\dbar}.
\end{align*}
\end{proof}

\begin{lemma}
For fixed $p,q\geq 0$, the map
$$\iota^*:H^{p,q}_{BC^{\beta}}(M)\rightarrow H^{p,q}_{\dbar^{\beta}}(M)$$
induced by inclusion is injective if and only if $(M,\mathcal{J})$ satisfies the $\partial^{\beta}\dbar^{\beta}$-Lemma.
\end{lemma}
\begin{proof}
Let $[\alpha]_{BC^{\beta}}\in H^{p,q}_{BC^{\beta}}(M)$ with $[\alpha]_{\dbar^{\beta}}=0$. Then $\iota^*$ is injective if and only if $[\alpha]_{BC^{\beta}}=0$, if and only if $\alpha=\partial^{\beta}\dbar^{\beta}\gamma$ for some $\gamma$. That is, for any $\alpha\in ker(\partial)\cap Im(\dbar)\cap\mathcal{A}^{p,q}$, we have $\alpha\in Im(\partial^{\beta}\dbar^{\beta})$, which is just the $\partial^{\beta}\dbar^{\beta}$-Lemma on $\mathcal{A}^{p,q}$.
\end{proof}

\begin{theorem}
On a compact generalized complex manifold $(M,\mathcal{J})$ of real dimension $2n$, the equality
$$Gh^k_{BC}=Gh^k_{\dbar}$$
holds for all $k \in [-n,n] \cap \bZ$ if and only if $M$ satisfies the generalized $\partial\dbar$-Lemma.
\end{theorem}
\begin{proof}
If the generalized $\partial\dbar$-Lemma holds, then the inclusion induces an isomorphism $GH^k_{BC}(M)\cong GH^k_{\dbar}(M)$, and hence $Gh^k_{BC}=Gh^k_{\dbar}$ for every $k \in [-n,n] \cap \bZ$.

To prove the converse, we make use of the bigrading introduced in Section \ref{subsec:bigrading}; similar arguments were used in \cite{CHT13}.
We first fix $k$. Then, for any $p,q \in \bN$ such that $p-q=k$, we define the following maps
\begin{align*}
\phi^{p,q}_+ & :H^{p,q}_{BC^{\beta}}(M)\rightarrow H^{p,q}_{\partial^{\beta}}(M),\\
\phi^{p,q}_- & :H^{p,q}_{BC^{\beta}}(M)\rightarrow H^{p,q}_{\dbar^{\beta}}(M),
\end{align*}
where $\phi^{p,q}_{\pm}$ are induced by inclusions.

We claim that, for fixed $p,q\geq 0$, if $\phi^{p,q}_+$ is injective, then $\phi^{p-1,q}_-$ is surjective and, if $\phi^{p,q}_-$ is injective, then $\phi^{p,q-1}_+$ is surjective. To prove the first assertion, we suppose that $\phi^{p,q}_+$ is injective. Then the $\partial^{\beta}\dbar^{\beta}$-Lemma holds on $\mathcal{A}^{p,q}$. Pick any $\alpha\in ker(\dbar^{\beta})\cap\mathcal{A}^{p-1,q}$. Then we have
$$\partial^{\beta}\alpha\in \text{im }\partial^{\beta}\cap \text{ker }\dbar^{\beta} \cap \mathcal{A}^{p,q} = \text{im }\partial^{\beta}\dbar^{\beta} \cap \mathcal{A}^{p,q}.$$
Hence there exists $\gamma\in\mathcal{A}^{p-1,q-1}$ such that $\partial^{\beta}\alpha=\partial^{\beta}\dbar^{\beta}\gamma$. Therefore, $\alpha-\dbar^{\beta}\gamma\in \text{ker }\partial^{\beta} \cap \text{ker }\dbar^{\beta} \cap \mathcal{A}^{p-1,q}$ and
$$\phi^{p-1,q}_-([\alpha-\dbar^{\beta}\gamma]_{BC^{\beta}})=[\alpha-\dbar^{\beta}\gamma]_{\dbar^{\beta}}=[\alpha]_{\dbar^{\beta}},$$
which proves the surjectivity of $\phi^{p-1,q}_-$.
The second assertion can be proved in exactly the same way.

Now suppose that $Gh^k_{BC}=Gh^k_{\dbar}$. Then we have
$$h^{p,q}_{BC^{\beta}}=Gh^k_{BC}=Gh^k_{\dbar}=h^{p,q}_{\dbar^{\beta}}=h^{p,q}_{\partial^{\beta}}.$$
We need to show that the map $\phi^{p,q}_-$ is an isomorphism for all $p,q\geq 0$, and we will prove this by contradiction.

So suppose that $\phi^{r,s}_-$ is not an isomorphism for some $r,s\geq 0$. Then $\phi^{r,s}_-$ is neither injective nor surjective by our dimension assumption.
Since $\phi^{r,s}_-$ is not surjective, our claim above shows that $\phi^{r+1,s}_+$ is not injective and hence not surjective.
Applying the claim again, we see that $\phi^{r+1,s+1}_-$ is again neither injective nor surjective. By induction, this implies that $\phi^{r+j,s+j}_-$ cannot be an isomorphism for all $j\geq 0$. However, this is impossible since $H^{p,q}_{\sharp^\beta}(M)$ is trivial for $p,q \gg 0$.

We thus conclude that $\phi^{p,q}_-$ is an isomorphism for all $p,q\geq 0$.
Therefore, the $\partial^{\beta}\dbar^{\beta}$-Lemma holds on $\mathcal{A}^{p,q}$ for all $p,q\geq 0$, which is equivalent to the validity of the generalized $\partial\dbar$-Lemma by Lemma \ref{lem:ddbar_vs_ddbarbeta}.
\end{proof}

\bibliographystyle{amsplain}
\bibliography{geometry}

\end{document}